\documentclass[12pt]{amsart}
\usepackage{amsmath,amssymb}
\usepackage{amsfonts}
\usepackage{amsthm}
\usepackage{latexsym}
\usepackage{graphicx}
\def\p{\partial}

\def\R{\mathbb{R}}

\def\vv<#1>{\langle#1\rangle}
\def\ol{\overline}

\def\1{\mathbf{1}}

\def\Comb{{\rm Comb}}

\def\XXint#1#2{\setbox0=\hbox{$#1{#2}{\int}$}{#2}\kern-.5\wd0 }

\def\XXint#1#2#3{{\setbox0=\hbox{$#1{#2#3}{\int}$}
     \vcenter{\hbox{$#2#3$}}\kern-.5\wd0}}

\def\vv<#1>{{\left\langle#1\right\rangle}}

\def\wh{\widehat}
\def\diam{{\rm diam}}
\def \St{{\rm St}}
\newtheorem{thm}{Theorem}[section]

\newtheorem{lem}{Lemma}[section]

\newtheorem{cor}{Corollary}[section]
\theoremstyle{definition}
\newtheorem{defn}{Definition}[section]
\theoremstyle{remark}

\newtheorem{rem}{Remark}[section]

\numberwithin{equation}{section}
\def \wt{\widetilde}
\def \spec{{\rm Spec}}
\begin{document}
\title{Monotonicity of Steklov eigenvalues on graphs and applications}

\author{Chengjie Yu$^1$}
\address{Department of Mathematics, Shantou University, Shantou, Guangdong, 515063, China}
\email{cjyu@stu.edu.cn}
\author{Yingtao Yu}
\address{Department of Mathematics, Shantou University, Shantou, Guangdong, 515063, China}
\email{18ytyu@stu.edu.cn}
\thanks{$^1$Research partially supported by GDNSF with contract no. 2021A1515010264 and NNSF of China with contract no. 11571215.}
\renewcommand{\subjclassname}{%
  \textup{2010} Mathematics Subject Classification}
\date{}
\begin{abstract}
In this paper, we obtain monotonicity of Steklov eigenvalues on graphs which as a special case on trees extends the results of He-Hua [Calc. Var. Partial Differential Equations 61 (2022), no. 3, Paper No. 101, arXiv: 2103.07696] to higher Steklov eigenvalues and gives affirmative answers to two problems proposed in He-Hua [arXiv: 2103.07696]. As applications of the monotonicity of Steklov eigenvalues, we obtain some estimates for Steklov eigenvalues on trees generalizing the isodiametric estimate for the first positive Steklov eigenvalues on trees in He-Hua [arXiv:2011.11014].
\end{abstract}
\maketitle\markboth{C. Yu \& Y. Yu}{Monotonicity of Steklov eigenvalues}
\section{Introduction}
The Steklov eigenvalues and Steklov operators for bounded domains in Euclidean spaces  were first introduced by Steklov \cite{St} when considering the problem of liquid sloshing. Such kinds of notions were later found important both in applied and pure mathematics. In applied mathematics, Steklov operators are the original models for the problem of detecting the inside of a body by boundary measurements. In pure mathematics, Steklov eigenvalues and Steklov eigenfunctions were found to be deeply related to free boundary minimal submanifolds in Euclidean balls (\cite{FS}) and the Yamabe problem on Riemannian manifolds with boundary (\cite{Es}). The discrete version of Steklov operators and Steklov eigenvalues were recently introduced by Hua-Huang-Wang \cite{HHW} and Hassannezhad-Miclo \cite{HM} independently. There are quite a number of works on exploring the properties Steklov eigenvalues on graphs in recent years. See for examples \cite{HH,He-Hua1,He-Hua2,HHW2,Pe1,Pe2, SY1,SY2,SY3}.

In a recent interesting work \cite{He-Hua2}, He and Hua obtained a monotonicity of the first positive Steklov eigenvalues on trees equipped with the unit weight. More precisely, they showed that
\begin{equation}\label{eq-He-Hua}
\sigma_2(G_2)\leq \sigma_2(G_1)
\end{equation}
if $G_1$ is a nontrivial subtree of $G_2$. Here a nontrivial tree is regarded as a graph with boundary intrinsically by regarding its leaves (vertices of degree one) as boundary vertices. Such a monotonicity of Steklov eigenvalues generalizes the isodiametric estimate
\begin{equation}\label{eq-isodiametric}
\sigma_2(G)\leq \frac{2}{\diam(G)}
\end{equation}
obtained in another interesting work \cite{He-Hua1} by them, in a  significant  way. Here $G$ is also a tree and $\diam(G)$ means the diameter of $G$. To obtain \eqref{eq-He-Hua}, He and Hua established a theory of Steklov flows in \cite{He-Hua2} and as an application of their theory, they also characterized the rigidity of the isodiametric estimate \eqref{eq-isodiametric}. In their work \cite{He-Hua2}, He and Hua also proposed two interesting problems about extending \eqref{eq-He-Hua} to higher Steklov eigenvalues and to general graphs. In this paper, we give affirmative answers to the two problems. In fact, we obtained monotonicity of Steklov eigenvalues for general weighted graphs. Our method is different with that of He and Hua in \cite{He-Hua2}.

Before stating our main results, we first recall the notion of weighted graphs.
\begin{defn}
A triple $(G,m,w)$ is called a weighted graph where
\begin{enumerate}
\item $G$ is a simple graph with $V(G)$ and $E(G)$ its sets of vertices and edges respectively;
\item $m:V(G)\to \R^+$ is called the vertex-measure;
\item $w:E(G)\to \R^+$ is called the edge-weight.
\end{enumerate}
If $m\equiv 1$ and $w\equiv 1$, $(G,m,w)$ is called a graph with unit weight. For convenience, the edge-weight $w$ is conventionally viewed as a symmetric function on $V(G)\times V(G)$ by zero extension. For a subgraph of $G$, we consider it as a weighted graph by restricting the vertex-measure and edge-weight of $G$ to its sets of vertices and edges respectively.
\end{defn}
Next, we intorudce the notions of graphs with boundary and combinatorial graphs with boundary.
\begin{defn}
\begin{enumerate}
\item A pair $(G,B)$ is called a graph with boundary if $G$ is a simple graph and $B\subset V(G)$. The sets $B$ and $\Omega:=V(G)\setminus B$ are called the boundary and interior of $(G,B)$ respectively. We  will also denote $B$ and $\Omega$ as $B(G)$ and $\Omega(G)$ respectively if necessary.
\item A simple graph $G$ is called a combinatorial graph with boundary if $G$ is equipped with the unit weight and $B(G)=\{v\in V(G)\ |\deg v\leq 1\}$. For a nontrivial combinatorial tree $G$ with boundary,  $B(G)=L(G)$ where $L(G)$ is the collection of leaves in $G$.
\end{enumerate}
\end{defn}
Note that the notion of graphs with boundary above is more general than the definition in \cite{Pe1}. We don't require that $E(B,B)=\emptyset$ and each boundary vertex is adjacent to some interior vertices. We drop the two requirements because they are not necessary in our discussion and the path with two vertices can not be viewed as a combinatorial graph with boundary if we take the definition in \cite{Pe1}. In fact, Steklov eigenvalues of such kind of more general graphs with boundary was considered in \cite{CG14,CG18} when considering discretization of Riemannian manifolds.

Similarly as in the smooth case, for a weighted finite graph $(G,B,m,w)$ with boundary, one can define the Steklov operator $\Lambda:\R^B\to\R^B$ (See Section \ref{sec-pre} for details) which is nonnegative and self-adjoint. The eigenvalues of $\Lambda$ are called the Steklov eigenvalues of the graph and are denoted as
$$0=\sigma_1(G,B,m,w)\leq \sigma_2(G,B,m,w)\leq\cdots\leq\sigma_{|B|}(G,B,m,w).$$
For simplicity, we will also write $\sigma_i(G,B,m,w)$ as $\sigma_i$, $\sigma_i(G)$ or $\sigma_i(G,B)$ if the ignored information is clear in context. When $i>|B|$, we conventionally take $\sigma_i=+\infty$. For the trivial graph $G$ with only one vertex, we view the vertex as a boundary vertex by definition and hence $\sigma_1(G)=0$ and $\sigma_i(G)=+\infty$ for $i\geq 2$.

The key ingredient of this work is the following extension of the notion of comb in \cite{FR}.
\begin{defn}A connected graph $\wt G$ is called a comb over its connected subgraph $G$ if  the graph $\wt G$ will break into $|V(G)|$ connected components when all the edges of $G$ are deleted. $G$ is then called the base of the  comb.

Suppose $\wt G$ is a comb over $G$.
\begin{enumerate}
 \item Denote the connected component of $\wt G$ containing $x\in V(G)$ as $\wt G_x$ when all the edges of $G$ are deleted, and call $\wt G_x$  the tooth of the comb  at $x$.
 \item When $\wt G$ is a graph with boundary $\wt B$, define $\wt B_x:=\wt B\cap V(\wt G_x)$ for any $x\in V(G)$.
\end{enumerate}
\end{defn}
We are now ready to state the first main result of this paper.

\begin{thm}\label{thm-main}
Let $(\wt G,\wt B, m, w)$ be a weighted connected finite graph with boundary, and $(G,B)$ be such that $G$ is a connected subgraph of $\wt G$. Suppose that $\wt G$ is a  comb over $G$, and
\begin{equation}\label{eq-assumption-boundary}
m(\wt B_x)\geq m_x
\end{equation}
for any $x\in B$. Then, for any $i=1,2,\cdots,|B|$,
\begin{equation}\label{eq-mono-main}
\sigma_i(\wt G)\leq \sigma_i(G).
\end{equation}
Here $m(\wt B_x)$ means the total measure of the vertices in $\wt B_x$.
\end{thm}

Note that a tree is a  comb over its subtrees. So, Theorem \ref{thm-main} can be applied directly to trees. When $G$ is a nontrivial subtree of the finite tree $\wt G$, because each tooth of $\wt G$ as a  comb over $G$ is also a tree,
\begin{equation}
|L(\wt G)\cap V(\wt G_x)|\geq 1
\end{equation}
for any $x\in L(G)$. This means that the assumption \eqref{eq-assumption-boundary} in Theorem \ref{thm-main} will be automatically satisfied on finite combinatorial trees. Therefore, we have the following corollary for combinatorial trees with boundary which gives an affirmative answer to Problem 1.5 in \cite{He-Hua2}.
\begin{cor}
Let $\wt G$ be a finite combinatorial tree with boundary and $G$ be a nontrivial subtree of $\wt G$ which is also viewed as a combinatorial tree with boundary. Then,
\begin{equation}
\sigma_i(\wt G)\leq \sigma_i(G)
\end{equation}
for $i=1,2,\cdots, |L(G)|$.
\end{cor}
More generally, we can extend the result above to combinatorial graphs with boundary which gives an affirmative answer to Problem 1.6 in \cite{He-Hua2}.
\begin{cor}\label{cor-mono-graph}
Let $\wt G$ be a connected finite combinatorial graph with boundary and $G$ be a connected nontrivial subgraph of $\wt G$ which is also viewed as the a combinatorial graph with boundary. Suppose that the inclusion map from $G$ to $\wt G$ induces an isomorphism on the fundamental groups of $G$ and $\wt G$. Then
\begin{equation}
\sigma_i(\wt G)\leq \sigma_i(G)
\end{equation}
for all $i=1,2,\cdots,|B(G)|$.
\end{cor}

The second part of this paper is devoted to consider the rigidity of \eqref{eq-mono-main} in Theorem \ref{thm-main}. In the following, the meanings of the symbols not mentioned before can be found in Section \ref{sec-pre}.

We first have the following general rigidity when \eqref{eq-mono-main} holds for all $i=1,2,\cdots,|B|$.
\begin{thm}\label{thm-main-rigidity}
Let $(\wt G,\wt B, m, w)$ be a weighted connected finite graph with boundary and $(G,B)$ be such that $G$ is a connected subgraph of $\wt G$ and $|B|\geq 2$. Suppose that $\wt G$ is a comb over $G$ and $$m(\wt B_x)\geq m_x$$ for any $x\in B$. Then, the equality of \eqref{eq-mono-main} holds for all $i=1,2,\cdots, |B|$, if and only if all the following statements are true:
\begin{enumerate}
\item for any $x\in B$, $\wt B_x=\{x\}$;
\item for any $y\in \Omega\setminus Z$, $\wt B_y=\emptyset$;
\item for any $u\in \R^{V(\wt G)}$ with $u|_B$ being constant and $\vv<u,1>_{\wt B}=0$,
\begin{equation}
\vv<du,du>_{\wt G}\geq \sigma_{|B|}(G)\vv<u,u>_{\wt B}.
\end{equation}
\end{enumerate}
Here $\Omega$ is the interior of $G$ and
\begin{equation*}
Z=\{x\in \Omega\ |\ f(x)=0,\ \forall f\in \R^{V(G)}\mbox{ with } (\Delta_Gf)\big|_\Omega\equiv 0\mbox{ and }\vv<f,1>_{B}=0.\}
\end{equation*}
In particular, if $Z=\emptyset$, then the equality of \eqref{eq-mono-main} holds for all $i=1,2,\cdots, |B|$, if and only if $\wt B=B$.
\end{thm}
The geometric meaning of (3) in Theorem \ref{thm-main-rigidity} is unclear. As a corollary of Theorem \ref{thm-main-rigidity}, we have the following geometric necessary condition for the rigidity when the equalities of \eqref{eq-mono-main} holds for $i=1,2,\cdots, |B|$.

\begin{thm}\label{thm-rigidity-Geom}
Let $(\wt G,\wt B, m, w)$ be a weighted connected finite graph with boundary and $(G,B)$ be such that $G$ is a connected subgraph of $\wt G$ with $|B|\geq 2$. Suppose that $\wt G$ is a  comb over $G$ and $$m(\wt B_x)\geq m_x$$ for any $x\in B$. Moreover, suppose the equalities of \eqref{eq-mono-main} hold  for $i=1,2,\cdots, |B|$. Then, for any $z\in Z$,
\begin{equation}
\lambda_1(\wt G_z,\wt B_z\setminus\{z\},\{z\})\geq \sigma_{|B|}(G).
\end{equation}
Here $Z$ is the same as in  Theorem \ref{thm-main-rigidity} and $\lambda_1$ means the first Steklov eigenvalues with vanishing Dirichlet boundary data (See Section \ref{sec-pre} for details).
\end{thm}
It seems that the converse of the above theorem is not true and it is difficult to obtain satisfied necessary and sufficient condition for equalities of \eqref{eq-mono-main} to hold for all $i=1,2,\cdots, |B|$. Motivated by He-Hua's work \cite{He-Hua2}, we then consider rigidity for $\sigma_2(G)=\sigma_2(\wt G)$.
\begin{thm}\label{thm-main-rigidity-2}
Let $(\wt G,\wt B, m, w)$ be a weighted connected finite graph with boundary and $(G,B)$ be such that $G$ is a connected subgraph of $\wt G$ with $|B|\geq 2$. Suppose that $\wt G$ is a comb over $G$ and $$m(\wt B_x)\geq m_x$$ for any $x\in B$. Moreover, suppose $\sigma_2(\wt G)=\sigma_2(G)$ and $Z_1\subset \Omega$ where
$$Z_1=\{x\in V(G)\ |\  f(x)=0 \mbox{ for any eigenfunction
$f\in \R^{V(G)}$ of $\sigma_2(G)$}\}.$$
Then,
\begin{enumerate}
\item for any $x\in B$, $\wt B_x=\{x\}$;
\item for any $y\in \Omega\setminus Z_1$, $\wt B_y=\emptyset$;
\item for any $z\in Z_1$, $\lambda_1(\wt G_z,\wt B_z\setminus\{z\}, \{z\})\geq \sigma_2( G)$.
\end{enumerate}
Here $\Omega$ is the interior of $G$.
\end{thm}
Then, by using He-Hua's trick in \cite{He-Hua2}, we have the following necessary and sufficient condition for $\sigma_2(G)=\sigma_2(\wt G)$ when $G$ has some symmetries.
\begin{thm}\label{thm-rigidity-2-sym}
Let $(\Gamma,\Sigma)$ be a connected finite graph with nonempty boundary and $z$ be an interior vertex of $\Gamma$. Let $G=\vee_z^r\Gamma$ with boundary $B=\sqcup^r \Sigma$ where $r\geq 2$ and $(\wt G,\wt B,m,w)$ be a  comb over $G$ equipped with the unit weight such that
\begin{enumerate}
\item for any $x\in B$, $\wt G_x=\wt B_x=\{x\}$;
\item for any $y\in V(G)\setminus(\{z\}\cup B)$, $\wt B_y=\emptyset$;
\item $z\not\in \wt B_z$.
\end{enumerate}
Then, $\sigma_2(\wt G)=\sigma_2(G)$ if and only if
\begin{equation}\label{eq-lambda1-sigma2}
\lambda_1(\wt G_z,\wt B_z,\{z\})\geq \sigma_2(G).
\end{equation}
Here $\vee_z^r\Gamma$ means the wedge-sum of $r$ copies of $\Gamma$ at $z$ (See Section \ref{sec-pre} for details.).
\end{thm}

Similarly as before, by applying Theorem \ref{thm-main-rigidity}, Theorem \ref{thm-rigidity-Geom}, Theorem \ref{thm-main-rigidity-2} and Theorem \ref{thm-rigidity-2-sym} to combinatorial trees with boundary, we have the following conclusion  for rigidity of the monotonicity on combinatorial trees.
\begin{cor}\label{cor-rigidity-tree}
Let $\wt G$ be a finite combinatorial tree with boundary and $G$ be its nontrivial subtree which is also viewed as a combinatorial tree. Then,
\begin{enumerate}
\item $\sigma_i(\wt G)=\sigma_i(G)$ for any $i=1,2,\cdots, |L(G)|$ if and only if all the following statements are true:
\begin{enumerate}
  \item[(i)] for any $x\in V(G)\setminus Z$, $\wt G_x=\{x\}$ which implies that $L(G)\subset L(\wt G)$;
 \item[(ii)] for any $u\in \R^{V(\wt G)}$ with $u|_{L(G)}$ being constant and $\vv<u,1>_{L(\wt G)}=0$,
 \begin{equation*}
\vv<du,du>_{\wt G}\geq \sigma_{|L(G)|}(G)\vv<u,u>_{L(\wt G)}.
\end{equation*}
\end{enumerate}
In particular, if $Z=\emptyset$, then $\sigma_i(\wt G)=\sigma_i(G)$ for all $i=1,2,\cdots,|L(G)|$ if and only if $G=\wt G$;\\
\item  if $\sigma_i(\wt G)=\sigma_i(G)$ for all $i=1,2,\cdots, |L(G)|$, then
\begin{equation*}
\lambda_1(\wt G_z,\wt G_z\cap L(\wt G), \{z\})\geq \sigma_{|L(G)|}(G)
\end{equation*}
for any $z\in Z$;\\
 \item if $Z_1\subset \Omega$ and $\sigma_2(G)=\sigma_2(\wt G)$, then $\wt G_x=\{x\}$ for any $x\in V(G)\setminus Z_1$, and
\begin{equation*}
\lambda_1(\wt G_z,\wt G_z\cap L(\wt G), \{z\})\geq \sigma_{2}(G)
\end{equation*}
for any $z\in Z_1$;\\
 \item if $G=\vee_z^rT$ for some nontrivial finite tree $T$ with $z\in V(T)$ and $r\geq 2$, and  $\wt G_x=\{x\}$ for any $x\in V(G)\setminus\{z\}$, then $\sigma_2(G)=\sigma_2(\wt G)$ if and only if $$\lambda_1(\wt G_z,L(\wt G)\cap \wt G_z,\{z\})\geq \sigma_2(G).$$
     Here $T$ is viewed as a graph with boundary $B(T)=L(T)\setminus\{z\}$.
\end{enumerate}
Here $\Omega,Z$ are the same as in Theorem \ref{thm-main-rigidity} and $Z_1$ is the same as in Theorem \ref{thm-main-rigidity-2}.
\end{cor}
In the third part of this paper, as applications of the monotonicity of Steklov eigenvalues, we obtain some estimates on combinatorial trees with boundary which may be viewed as generalizations of He-Hua's isodiametric estimate in \cite{He-Hua1}. For example, among the  other estimates, we have the following results. For the definitions of star, regular star and regular comb, see \cite{FR} or Section \ref{sec-app}.
\begin{thm} Let $\wt G$ be a finite combinatorial graph with boundary.

(1)  If $\wt G$ contains the regular star $\St(r;l)$ as a subtree with $r\geq 2$ and $l\geq 1$, then
\begin{equation*}
\sigma_i(\wt G)\leq\frac1l
\end{equation*}
for $i=2,3,\cdots, r$. Moreover, the equality holds for  $i=2$ if and only if $\wt G=\St(r;l)\vee_o\wt G_o$ with $\lambda_1\left(\wt G_o, L(\wt G_o)\setminus \{o\}, o\right)\geq \frac{1}{l}$. Here $o$ is the center of the star $\St(r;l)$.

(2) If $\wt G$ contains the star $\St(l_1,l_2,\cdots,l_r)$ with $r\geq 2$ and $1\leq l_1\leq l_2\cdots\leq l_r$ as a subtree, then $$\sigma_2(\wt G)\leq \frac{r}{l_1+l_2+\cdots l_r}.$$ Moreover, when $r\geq 3$, the equality holds if and only if $l_1=l_2=\cdots=l_r$ and $\wt G=\St(r;l_1)\vee_o\wt G_o$ with $\lambda_1(\wt G_o,L(\wt G_o)\setminus\{o\},\{o\})\geq \frac{1}{l_1}$. Here $o$ is the center of the star.

(3) If $\wt G$ contains the regular comb $\Comb(r;l)$ with $r,l\geq 1$ as a subtree, then,
\begin{equation*}
\sigma_i(\wt G)\leq \frac{4\sin^2\frac{(i-1)\pi}{2(r+1)}}{1+4l\sin^2\frac{(i-1)\pi}{2(r+1)}}
\end{equation*}
for $i=1,2\cdots, r+1$. Moreover, the equality holds for all $i=1,2,\cdots, r+1$ if and only if $\wt G=\Comb(r;l)$.
\end{thm}

The rest of this paper is organized as follows. In Section \ref{sec-pre}, we introduce some preliminary notions and symbols. In Section \ref{sec-mono}, we prove Theorem \ref{thm-main}, Theorem \ref{thm-main-rigidity}, Theorem \ref{thm-rigidity-Geom} and their corollaries. In Section \ref{sec-rigid}, we prove Theorem \ref{thm-main-rigidity-2}, Theorem \ref{thm-rigidity-2-sym} and their corollaries. In Section \ref{sec-app}, we apply the monotonicity of Steklov eigenvalues to obtain estimates extending the isodiametric estimate of He-Hua \cite{He-Hua1}.
\section{Preliminaries}\label{sec-pre}
In this section, we introduce some notions and symbols that are used throughout the paper.

For a weighted finite graph $(G,m,w)$, let $A^1(G)$ be the collection of skew symmetric functions $\alpha:V(G)\times V(G)\to \R$ such that $\alpha(x,y)=0$ when $x\not\sim y$ (Such a function is called a flow in \cite{Ba}.). The exterior differential map $d: \R^{V(G)}\to A^1(G)$ is defined as
\begin{equation}
df(x,y)=\left\{\begin{array}{ll}f(y)-f(x)&x\sim y\\
0&x\not\sim y.
\end{array}\right.
\end{equation}
Moreover, equip with $\R^{V(G)}$ and $A^1(G)$ the following natural inner products
\begin{equation}
\vv<f,g>_{G}=\sum_{x\in V(G)}f(x)g(x)m_x
\end{equation}
and
\begin{equation}
\vv<\alpha,\beta>_{G}=\sum_{\{x,y\}\in E(G)}\alpha(x,y)\beta(x,y)w_{xy}
\end{equation}
for any $f,g\in \R^{V(G)}$ and $\alpha,\beta\in A^1(G)$. For $A\subset V(G)$, we denote
\begin{equation}
m(A)=\sum_{x\in A}m_x
\end{equation}
and
\begin{equation}
\vv<f,g>_{A}=\sum_{x\in A}f(x)g(x)m_x.
\end{equation}
Denote the adjoint operator of $d:\R^{V(G)}\to A^1(G)$ as $d^*$. Then, similar to the smooth case, define the Laplacian operator $\Delta_G:\R^{V(G)}\to \R^{V(G)}$ as
\begin{equation}
\Delta_G=-d^*d.
\end{equation}
A direct computation will give us
\begin{equation}
\Delta_G f(x)=\frac{1}{m_x}\sum_{y\sim x}(f(y)-f(x))w_{xy}.
\end{equation}
Moreover, by the definition of $\Delta_G$, it is clear that
\begin{equation}\label{eq-integrate-by-parts}
\vv<\Delta_Gf,g>_G=-\vv<df,dg>_{G}\ \ \forall f,g\in \R^{V(G)}.
\end{equation}
So, $-\Delta_G$ is a nonnegative self-adjoint operator on $\R^{V(G)}$. Its eigenvalues:
$$0=\mu_1(G,m,w)\leq\mu_2(G,m,w)\leq\cdots\leq \mu_{|V(G)|}(G,m,w)$$
are called Laplacian eigenvalues of $(G,m,w)$. For simplicity, we will also write $\mu_i(G,m,w)$ as $\mu_i(G)$ or $\mu_i$ if the ignored information is clear in context.

Let $(G,B,m,w)$ be a weighted finite graph with boundary and $f\in \R^{V(G)}$, define the outward normal derivative of $f$ at $x\in B$ as
\begin{equation}\label{eq-normal-derivative}
\frac{\p f}{\p n}(x):=\frac{1}{m_x}\sum_{y\in V(G)}(f(x)-f(y))w_{xy}=-\Delta_Gf(x).
\end{equation}
Then, by \eqref{eq-integrate-by-parts}, it is clear that
\begin{equation}\label{eq-Green}
\vv<df,dg>_G=-\vv<\Delta_Gf,g>_\Omega+\vv<\frac{\p f}{\p n},g>_B.
\end{equation}
This is a discrete version of Green's formula and is also the reason one defines $\frac{\p f}{\p n}$ as in \eqref{eq-normal-derivative}.

Similar to the smooth case, a real number $\sigma$ is called a Steklov eigenvalue of $(G,B,m,w)$ if the following boundary value problem:
\begin{equation}
\left\{\begin{array}{ll}\Delta_G f(x)=0&x\in\Omega\\
\frac{\p f}{\p n}(x)=\sigma f(x)&x\in B
\end{array}\right.
\end{equation}
has a nonzero solution $f$. The nonzero solution $f$ is called a Steklov eigenfunction of the graph for $\sigma$. Steklov eigenvalues are the eigenvalues of the so called Steklov operator or Dirichlet-to-Neumann map (DtN map for short):
\begin{equation}
\Lambda:\R^B\to \R^B, \Lambda(u)=\frac{\p \wh u}{\p n}
\end{equation}
where $\wh u$ is the harmonic extension of $u$. By \eqref{eq-Green},
\begin{equation}
\vv<\Lambda(u),v>_B=\vv<d\wh u,d\wh v>_G
\end{equation}
for any $u,v\in \R^{B}$. So, the the corresponding Rayleigh quotient for Steklov eigenvalues is
\begin{equation}
R[u]=\frac{\vv<d\wh u,d\wh u>_G}{\vv<u,u>_B},
\end{equation}
 and the Steklov eigenvalues are nonnegative. We denote them as
\begin{equation*}
0=\sigma_1(G,B,m,w)\leq \sigma_2(G,B,m,w)\leq\cdots\leq \sigma_{|B|}(G,B,m,w).
\end{equation*}
The first Steklov eigenvalue is zero because constant functions are its eigenfunctions. The second Steklov eigenvalue will be positive if $G$ is assumed to be connected (See \cite{HHW}). For simplicity, we will also simply denote $\sigma_i(G,B,m,w)$ as $\sigma_i$, $\sigma_i(G)$ or $\sigma_i(G,B)$ if the ignored information is clear in context.

As in \cite{He-Hua1}, we also need the notion of Steklov eigenvalues with vanishing Dirichlet boundary data. Let $(G,B,m,w)$ be a weighted finite graph with boundary and $Z\subset \Omega$. A real number $\lambda$ is called a Steklov eigenvalues for $(G,B,m,w)$ with vanishing Dirichlet boundary data on $Z$ if the following boundary value problem:
\begin{equation}
\left\{\begin{array}{ll}\Delta_G f(x)=0&x\in\Omega\setminus Z;\\
\frac{\p f}{\p n}(x)=\lambda f(x)&x\in B\\
f(x)=0&x\in Z
\end{array}\right.
\end{equation}
has a nonzero solution $f$. The  eigenvalue $\lambda$ is an eigenvalue of the so called DtN map with vanishing Dirichlet boundary data on $Z$:
\begin{equation}
\Lambda_0:\R^{B}\to\R^B, \Lambda_0(u)=\frac{\p \zeta_u}{\p n}
\end{equation}
where $\zeta_u$ is the solution of
\begin{equation}
\left\{\begin{array}{ll}\Delta_G \zeta_u(x)=0&x\in\Omega\setminus Z\\
\zeta_u(x)=u(x)&x\in B\\
\zeta_u(x)=0&x\in Z.
\end{array}\right.
\end{equation}
By \eqref{eq-Green},
\begin{equation}
\vv<\Lambda_0(u),v>_{B}=\vv<d\zeta_u,d\zeta_v>_G
\end{equation}
for any $u,v\in \R^{B}$. So the corresponding Rayleigh quotient for Steklov eigenvalues with vanishing Dirichlet boundary data on $Z$ is
\begin{equation}
R_0[u]=\frac{\vv<d\zeta_u,d\zeta_v>_G}{\vv<u,u>_B}
\end{equation}
and if the graph is connected, all the eigenvalues are positive. We denote them as
\begin{equation*}
\lambda_1(G,B,Z)\leq \lambda_2(G,B,Z)\leq\cdots\leq \lambda_{|B|}(G,B,Z).
\end{equation*}
The collection of all the eigenvalues counting multiplicity is denoted as $\spec(G,B,Z)$. By the Dirichlet principle, it is clear that
\begin{equation}\label{eq-lambda1}
\lambda_1(G,B,Z)=\inf\left\{\frac{\vv<du,du>_G}{\vv<u,u>_B}\ \bigg|\ u\in \R^{V(G)}\mbox{ with } u|_B\not\equiv 0\mbox{ and }u|_Z\equiv0.\right\}.
\end{equation}
If $B=\emptyset$, we take $\lambda_1(G,B,Z)=+\infty$ conventionally.

Next recall the following notion of wedge-sum of graphs (see also \cite{He-Hua2}).
\begin{defn}
Let $(G_1,B_1)$ and $(G_2,B_2)$ be two graphs with boundary and $z_i\in \Omega(G_i)$ for $i=1,2$. Then,  the wedge-sum of $(G_1,B_1)$ and $(G_2, B_2)$ at $z_1$ and $z_2$ is denoted as $(G_1\vee_{z_1,z_2}G_2, B_1\sqcup B_2)$ with
\begin{equation}
V(G_1\vee_{z_1,z_2}G_2)=V(G_1)\sqcup V(G_2)/z_1\backsim z_2
\end{equation}
and
\begin{equation}
E(G_1\vee_{z_1,z_2}G_2)=E(G_1)\sqcup E(G_2).
\end{equation}
Here $\sqcup$ means disjoint union and $z_1\backsim z_2$ means identifying $z_1$ and $z_2$. We also simply denote the wedge-sum of $(G_1,B_1)$ and $(G_2,B_2)$ at $z_1$ and $z_2$ as $G_1\vee_{z_1,z_2}G_2$. The wedge-sum of two copies of $(G,B)$ at $z\in \Omega(G)$ is denoted as $G\vee_zG$. The wedge-sum of $r$-copies of $(G,B)$ at $z\in\Omega(G)$ is denoted as $\vee_z^rG$. If $G$ is weighted, then the weight of $\vee_z^rG$ is inherited from $G$ in the natural way.
\end{defn}
Finally, we need the following notion of constant extensions of a function defined on the base of a comb.
\begin{defn}Let $\wt G$ be a comb over $G$ and $f\in \R^{V(G)}$, the constant extension $\wt f\in \R^{V(\wt G)}$ of $f$ is defined as $\wt f|_{V(\wt G_x)}\equiv f(x)$ for any $x\in V(G)$.

\end{defn}

\section{Monotonicity of Steklov eigenvalues and its rigidity}\label{sec-mono}
In this section, we prove Theorem \ref{thm-main}, Theorem \ref{thm-main-rigidity}, Theorem \ref{thm-rigidity-Geom} and their corollaries. We first give the proof of Theorem \ref{thm-main}. Our method uses only the min-max principle of Courant and is different with that of He-Hua \cite{He-Hua2}.
\begin{proof}[Proof of Theorem \ref{thm-main}]
Let $\varphi_1=1, \varphi_2,\cdots, \varphi_{|\wt B|}$ be an orthogonal system of Steklov eigenfunctions on $\wt G$ such that $\varphi_i$ is the Steklov eigenfunction for $\sigma_i(\wt G)$ for $i=1,2,\cdots, |\wt B|$. Let $f_1=1,f_2,\cdots, f_{|B|}$ be an orthogonal system of Steklov eigenfunctions on $G$ such that $f_i$ is the Steklov eigenfunction for $\sigma_i(G)$ for $i=1,2,\cdots,|B|$. For each $i\geq 2$, let
$$f=c_1f_1+c_2f_2+\cdots+c_{i}f_i$$
where $c_1,c_2,\cdots,c_i$ are constants that are not all zero such that
\begin{equation}\label{eq-orthogonal}
\vv<\wt f, \varphi_j>_{\wt B}=0
\end{equation}
for any $j=1,2,\cdots,i-1$. The existence of $c_1,c_2,\cdots,c_i$ is clear because \eqref{eq-orthogonal} with $j=1,2,\cdots,i-1$ form a homogeneous linear system with $i-1$ equations on $i$ unknowns $c_1,c_2,\cdots,c_i$ which certainly has nonzero solutions. Then, by using the assumption \eqref{eq-assumption-boundary},
\begin{equation}\label{eq-sigma-i}
\begin{split}
\sigma_i(\wt G)\leq \frac{\vv<d\wt f,d\wt f>_{\wt G}}{\vv<\wt f,\wt f>_{\wt B}}
=&\frac{\vv<df,df>_G}{\sum_{x\in V(G)}f^2(x)m(\wt B_x)}\\ \leq&\frac{\vv<df,df>_G}{\sum_{x\in B}f^2(x)m_x}=\frac{\vv<df,df>_G}{\vv<f,f>_{B}}\leq \sigma_i(G).
\end{split}
\end{equation}
This completes the proof of the theorem.
\end{proof}
We next prove  Corollary \ref{cor-mono-graph}.
\begin{proof}[Proof of Corollary \ref{cor-mono-graph}]
Because the inclusion map from $G$ to $\wt G$ induces an isomorphism on fundamental groups of $G$ and $\wt G$, any cycle in $\wt G$ is contained in $G$. Thus, $\wt G$ is a comb over $G$, and for any $x\in V(G)$, the tooth $\wt G_x$ is a tree. Therefore,
\begin{equation}
|V(\wt G_x)\cap B(\wt G)|\geq 1
\end{equation}
for any $x\in B(G)$. So, the graphs $(\wt G, \wt B:=B(\wt G) )$ and $(G, B:=B(G))$ equipped with the unit weight satisfy \eqref{eq-assumption-boundary} in Theorem \ref{thm-main}. Then, the conclusion follows from Theorem \ref{thm-main}.
\end{proof}
Next, we come to prove Theorem \ref{thm-main-rigidity} considering rigidity of the monotonicity in Theorem \ref{thm-main} in general situation.
\begin{proof}[Proof of Theorem \ref{thm-main-rigidity}]
We first show that \emph{there are $|B|$ functions $g_1\equiv 1,g_2,\cdots, g_{|B|}$ on $G$ such that
\begin{enumerate}
\item[(i)] each $g_i$ is a Steklov eigenfunction for $\sigma_i(G)$;
\item[(ii)] each $\wt g_i$ is a Steklov eigenfunction for $\sigma_i(\wt G)=\sigma_i(G)$;
\item[(iii)] for any $x\in \Omega$, $g_i^2(x)m(\wt B_x)=0$ when $i\geq 2$;
\item[(iv)] $\vv<\wt g_i,\wt g_j>_{\wt B}=0$ for any $1\leq i<j\leq |B|$,
\end{enumerate}
}
\noindent by induction. In fact, for $i\geq 2$, suppose $g_1\equiv 1,g_2,\cdots, g_{i-1}$ satisfying the above properties have been constructed. Let $$g_i=c_1f_1+c_2f_2+\cdots+c_i f_i$$ with $c_1,c_2\cdots,c_i$ constants to be determined such that
\begin{equation}\label{eq-orthogonal-wt-g}
\vv<\wt g_i,\wt g_j>_{\wt B}=0
\end{equation}
for $j=1,2,\cdots, i-1$. Here $f_1,f_2,\cdots, f_{|B|}$ are the same as in the proof of Theorem \ref{thm-main}. The same reason as in the proof of Theorem \ref{thm-main}, we can find $c_1,c_2,\cdots,c_i$ not all zero satisfying \eqref{eq-orthogonal-wt-g}. Then, the same argument as in \eqref{eq-sigma-i} in the proof of Theorem \ref{thm-main}, we know that $g_i$ must satisfy the properties (i)--(iv) above since $\sigma_i(\wt G)=\sigma_i(G)$.

By (iv), we know that $g_1,\cdots,g_{|B|}\in \R^{V(G)}$ are linearly independent. Because $g_i$ is the harmonic extension of $g_i|_B$ for $i=1,2,\cdots,|B|$, $g_1|_B,g_2|_B,\cdots, g_{|B|}|_B$ form a basis for $\R^{B}$. This implies that for any $x\in B$, there is some $g_i$ with $i\geq 2$ such that $g_i(x)\neq 0$. Now, if there is some $y\in \wt B_x\setminus \{x\}$, then
\begin{equation}
0=\frac{\p \wt g_i}{\p n}(y)=\sigma_i(\wt G)\wt g_i(y)=\sigma_i(\wt G)g_i(x)
\end{equation}
which implies that $\sigma_i(\wt G)=0$ and is a contradiction because $\wt G$ is assumed to be connected. Therefore, $\wt B_x\setminus \{x\}=\emptyset$  and hence $\wt B_x=\{x\}$ since $\wt B_x \neq\emptyset$ for all $x\in B$ by assumption \eqref{eq-assumption-boundary}. This proves (1).

Moreover, for any $y\in \Omega\setminus Z$, there is some $g_i$ with $i\geq 2$ such that $g_i(y)\neq 0$. Then, by (iii), $m(\wt B_y)=0$. Thus $\wt B_y=\emptyset$. This proves (2).

Conversely, it is clear that if (1) and (2) holds, then for any Steklov eigenfunction $f\in \R^{V(G)}$ of $G$, $\wt f$ is also a Steklov eigenfunction of $\wt G$ with the same eigenvalue. Thus, the equality of \eqref{eq-mono-main} holds for all $i=1,2,\cdots, |B|$ if and only if for any $u\in \R^{V(\wt G)}$ with
\begin{equation}
\vv<u,\wt g_i>_{\wt B}=0
\end{equation}
for $i=1,2,\cdots,|B|$,
\begin{equation}\label{eq-rigidity-equiv}
\vv<du,du>_{\wt G}\geq \sigma_{|B|}(G)\vv<u,u>_{\wt B}.
\end{equation}
By (1) and (2),
\begin{equation}\label{eq-u-g-i}
0=\vv<u,\wt g_i>_{\wt B}=\vv<u,g_i>_B
\end{equation}
for $i=2,\cdots,|B|$. So, $u|_B$ must be constant and $\vv<u,1>_{\wt B}=0$. Conversely, for any $u\in \R^{V(\wt G)}$ with $u|_{B}$ being constant and $\vv<u,1>_{\wt B}=0$, it is clear that $\vv<u,\wt g_i>_{\wt B}=0$ for $i=1,2,\cdots, |B|$. This completes the proof of the theorem.
\end{proof}

Finally, we come to prove Theorem \ref{thm-rigidity-Geom}. Before proving it, we need the following lemma on Steklov eigenvalues with vanishing Dirichlet boundary data.
\begin{lem}\label{lem-DtN-dirichlet}
Let $(G,B,m,w)$ be a weighted connected finite graph with boundary and $Z$ be a subset of $\Omega$. Then,
\begin{enumerate}
\item there is an eigenfunction $u$ of $\lambda_1(G,B,Z)$ such that $u\geq 0$;
\item an eigenfunction $u$ of the Steklov operator on $(G,B)$ with vanishing Dirichlet boundary data on $Z$ and with $u|_B>0$ must be an eigenfunction for $\lambda_1(G,B,Z)$.
\end{enumerate}
\end{lem}
\begin{proof}
\begin{enumerate}
\item Let $f$ be an eigenfunction for $\lambda_1(G,B,Z)$. Then, by \eqref{eq-lambda1},
\begin{equation}
\lambda_1(G,B,Z)=\frac{\vv<df,df>_{G}}{\vv<f,f>_B}\geq \frac{\vv<d|f|,d|f|>_G}{\vv<|f|,|f|>_{B}}\geq \lambda_1(G,B,Z).
\end{equation}
So, $u=|f|\geq 0$ is also an eigenfunction of $\lambda_1(G,B,Z)$.
\item Let $f\geq 0$ be an eigenfunction for $\lambda_1(G,B,Z)$. If $u$ is not an eigenfunction for $\lambda_1(G,B,Z)$, then
    \begin{equation}
    \vv<u,f>_{B}=0.
    \end{equation}
However, this is impossible since $u|_B>0$,$f|_{B}\geq0$ and $f|_{B}\not\equiv 0$.
\end{enumerate}
\end{proof}
We are ready to prove Theorem \ref{thm-rigidity-Geom}.
\begin{proof}[Proof of Theorem \ref{thm-rigidity-Geom}]
Let $g_1\equiv 1,g_2,\cdots,g_{|B|}$ be the Steklov eigenfunctions of $G$ found in the proof Theorem \ref{thm-main-rigidity}. Because $g_2|_Z=\cdots=g_{|B|}|_Z=0$, $g_2,\cdots,g_{|B|}$ are also eigenfunctions for the DtN map with vanishing Dirichlet boundary data on $Z$. Let $f$ be the solution of
\begin{equation}
\left\{\begin{array}{ll}\Delta_G f(x)=0&x\in\Omega\setminus Z\\
f(x)=1&x\in B\\
f(x)=0&x\in Z.
\end{array}\right.
\end{equation}
Then $\vv<f,g_i>_B=0$ for $i=2,3,\cdots,|B|$. Moreover, by Green's formula \eqref{eq-Green},
\begin{equation}
\begin{split}
\vv<\frac{\p f}{\p n}, g_i>_B=\vv<f,\frac{\p g_i}{\p n}>_B=\sigma_i(G)\vv<f, g_i>_B=0
\end{split}
\end{equation}
for $i=2,\cdots,n$. Thus $\frac{\p f}{\p n}\big|_{B}$ is constant and  $f$ is an eigenfunction for the DtN map with vanishing Dirichlet boundary data on $Z$. By (2) of Lemma \ref{lem-DtN-dirichlet}, $f$ is an eigenfunction for
$$\lambda_1(G,B,Z)\leq \sigma_2(G)$$
because $f|_B\equiv 1>0$.

Let $\varphi\in \R^{V(\wt G_z)}$ be an eigenfunction for $\lambda_1(\wt G_z,\wt B_z\setminus\{z\},\{z\})$ and $\ol\varphi\in \R^{V(\wt G)}$ be the zero extension of $\varphi$. Let $c$ be a constant such that
\begin{equation}
\vv<\ol\varphi+c\wt f,1>_{\wt B}=0.
\end{equation}
Then, by using Theorem \ref{thm-main-rigidity} with $u=\ol\varphi+c\wt f$,
\begin{equation}
\begin{split}
\sigma_{|B|}(G)\leq\frac{\vv<d(\ol\varphi+c\wt f),d(\ol\varphi+c\wt f)>_{\wt G}}{\vv<\ol\varphi+c\wt f,\ol\varphi+c\wt f>_{\wt B}}=\frac{\vv<d\varphi,d\varphi>_{\wt G_z}+c^2\vv<df,df>_{G}}{\vv<\varphi,\varphi>_{\wt B_z}+c^2\vv<f,f>_B}.
\end{split}
\end{equation}
Moreover, by that $$\vv<df,df>_{G}=\lambda_1(G,B,Z)\vv<f,f>_{B}\leq \sigma_{|B|}(G)\vv<f,f>_B,$$
we have
\begin{equation}
\lambda_1(\wt G_z,\wt B_z\setminus\{z\},\{z\})=\frac{\vv<d\varphi,d\varphi>_{\wt G_z}}{\vv<\varphi,\varphi>_{\wt B_z}}\geq \sigma_{|B|}(G).
\end{equation}
This completes the proof of the theorem.
\end{proof}
\section{Rigidity for first positive eigenvalues}\label{sec-rigid}
In this section, we consider the rigidity for $\sigma_2(G)=\sigma_2(\wt G)$, and prove Theorem \ref{thm-main-rigidity-2} and Theorem \ref{thm-rigidity-2-sym}. We first prove Theorem \ref{thm-main-rigidity-2}. The idea is similar with the proofs of Theorem \ref{thm-main-rigidity} and Theorem \ref{thm-rigidity-Geom}.
\begin{proof}[Proof of Theorem \ref{thm-main-rigidity-2}]
Let $f\in \R^{V(G)}$ be an eigenfunction of $\sigma_2(G)$ and  $c$ be a constant such that
\begin{equation*}
\vv<c+\wt f, 1>_{\wt B}=0.
\end{equation*}
Then,
\begin{equation*}
\begin{split}
\sigma_2(G)=\sigma_2(\wt G)\leq& \frac{\vv<d(c+\wt f),d(c+\wt f)>_{\wt G}}{\vv<c+\wt f,c+\wt f>_{\wt B}}\\
\leq& \frac{\vv<df,df>_G}{\vv<c+f,c+f>_{B}}=\frac{\vv<df,df>_G}{c^2m(B)+\vv<f,f>_{B}}\leq \frac{\vv<df,df>_G}{\vv<f,f>_{B}}=\sigma_2(G).
\end{split}
\end{equation*}
Thus $c=0$ and $\wt f$ is also an eigenfunction of $\sigma_2(\wt G)$ on $\wt G$. Now the same argument as in the proof of Theorem \ref{thm-main-rigidity} will give us (1) and (2).

Moreover, for $z\in Z_1$, let $\varphi\in \R^{V(\wt G_z)}$ be such that $\varphi|_{\wt B_z}\not\equiv 0$ and $\varphi(z)=0$. Let $\ol \varphi\in \R^{V(\wt G)}$ be the zero extension of $\varphi$. Let $f\in \R^{V(G)}$ be an eigenfunction for $\sigma_2(G)$ and $g=|f|$. Choose $c\in \R$ such that
\begin{equation}
\vv<\ol \varphi+c\wt g,1>_{\wt B}=0.
\end{equation}
Then,
\begin{equation}
\begin{split}
\sigma_2(G)=&\sigma_{2}(\wt G)\\
\leq &\frac{\vv<d(\ol \varphi+c\wt g),d(\ol \varphi+c\wt g)>_{\wt G}}{\vv<\ol \varphi+c\wt g,\ol \varphi+c\wt g>_{\wt B}}\\
=&\frac{\vv<d\varphi,d\varphi>_{\wt G_z}+c^2\vv<dg,dg>_G}{\vv<\varphi,\varphi>_{\wt B_z}+c^2\vv<g,g>_{B}}.
\end{split}
\end{equation}
Note that
\begin{equation}
\frac{\vv<dg,dg>_{G}}{\vv<g,g>_{B}}\leq \frac{\vv<df,df>_G}{\vv<f,f>_B}=\sigma_2(G).
\end{equation}
Thus
\begin{equation}
\frac{\vv<d\varphi,d\varphi>_{\wt G_z}}{\vv<\varphi,\varphi>_{\wt B_z}}\geq \sigma_2(G).
\end{equation}
This gives us (3) and the proof of the theorem is completed.
\end{proof}
Before proving   Theorem \ref{thm-rigidity-2-sym}, we need the following key result extending the corresponding result for trees in \cite{He-Hua2} to general weighted graphs.
\begin{thm}\label{thm-sym}
Let $(G,B,m,w)$ be a weighted connected finite graph with boundary and $z\in \Omega$. Then, any Steklov eigenfunction $f$ of $G\vee_z G$ for the eigenvalue $\sigma_2(G\vee_z G)$ must satisfy $f(z)=0$. Moreover
\begin{equation}
\sigma_2(G\vee_zG)=\lambda_1(G,B,\{z\}).
\end{equation}
\end{thm}
\begin{proof}
For convenience, we denote the second copy of $(G,B)$ in $G\vee_zG$ as $(G',B')$, and for any vertex $y\in V(G)\setminus \{z\}$, denote the corresponding vertex in $G'$ as $y'$. For any $u\in \R^{V(G\vee_zG)}$, let $u^s\in \R^{V(G\vee_zG)}$ be such that $u^s(z)=u(z)$, $u^s(y')=u(y)$ and $u^s(y)=u(y')$ for any $y\in V(G)\setminus\{z\}$.

For any Steklov eigenfunction $f$ of $G\vee_z G$ for the eigenvalue $\sigma_2(G\vee_z G)$, it is clear that $f^s$ is still an eigenfunction for $\sigma_2(G\vee_zG)$. If $$g=f+f^s\equiv 0,$$ then $$f(z)=\frac12g(z)=0,$$ and we are done. Otherwise, $g$ is an eigenfunction for $\sigma_2(G\vee_zG)$ with $$g^s=g.$$ Then,
\begin{equation}\label{eq-g-1}
\vv<g,1>_B=\vv<g,1>_{B'}=0
\end{equation}
and
\begin{equation}\label{eq-g-2}
\frac{\vv<dg,dg>_{G}}{\vv<g,g>_B}=\frac{\vv<dg,dg>_{G'}}{\vv<g,g>_{B'}}=\sigma_2(G\vee_zG).
\end{equation}
Let $\wt g\in \R^{V(G\vee_zG)}$ be such that $\wt g|_{V(G)}=g|_{V(G)}$ and $\wt g(y')=g(z)$ for any $y\in V(G)\setminus\{z\}$. Let $a\in \R$ be such that
$$\vv<a+\wt g,1>_{B\cup B'}=0.$$
Then, by \eqref{eq-g-1} and \eqref{eq-g-2},
\begin{equation}\label{eq-sigma-2-GvG}
\begin{split}
\sigma_2(G\vee_zG)\leq& \frac{\vv<d(a+\wt g),d(a+\wt g)>_{G\vee_zG}}{\vv<a+\wt g,a+\wt g>_{B\cup B'}}\\
=&\frac{\vv<dg,dg>_{G}}{\vv<a+ g,a+ g>_{B}+ \vv<a+\wt g,a+\wt g>_{B'}}\\
\leq&\frac{\vv<dg,dg>_{G}}{\vv<a+g,a+g>_{B}}\\
=&\frac{\vv<dg,dg>_{G}}{a^2m(B)+\vv< g, g>_{B}}\\
\leq&\frac{\vv<dg,dg>_{G}}{\vv< g, g>_{B}}\\
=&\sigma_2(G\vee_zG).
\end{split}
\end{equation}
Thus, all the inequalities in \eqref{eq-sigma-2-GvG} must be equalities which implies that $a=0$ and $g(z)=0$. Hence $f(z)=\frac12 g(z)=0$.

Moreover, for any eigenfunction $u\in \R^{V(G)}$ for $\lambda_1(G,B,\{z\})$, let $v\in \R^{V(G\vee_zG)}$ be such that $v|_{V(G)}=u$ and $v(y')=-u(y)$ for any $y\in V(G)\setminus\{z\}$. Then, $\vv<v,1>_{B\cup B'}=0$ and hence
\begin{equation}
\begin{split}
\sigma_2(G\vee_zG)\leq \frac{\vv<dv,dv>_{G\vee_zG}}{\vv<v,v>_{B\cup B'}}=\frac{\vv<du,du>_{G}}{\vv<u,u>_B}=\lambda_1(G,B,\{z\}).
\end{split}
\end{equation}
On the other hand, let $f\in \R^{V(G\vee_zG)}$ be an eigenfunction for $\sigma_2(G\vee_zG)$. Then
\begin{equation}
\begin{split}
\vv<df,df>_G+\vv<df^s,df^s>_{G}=&\vv<df,df>_{G\vee_zG}\\
=&\sigma_2(G\vee_z G)\vv<f,f>_{B\cup B'}\\
=&\sigma_2(G\vee_z G)(\vv<f,f>_{B}+\vv<f^s,f^s>_{B})
\end{split}
\end{equation}
Thus, either
\begin{equation}
\vv<df,df>_G\leq \sigma_2(G\vee_z G) \vv<f,f>_{B}
\end{equation}
or
\begin{equation}
\vv<df^s,df^s>_G\leq \sigma_2(G\vee_z G) \vv<f^s,f^s>_{B}.
\end{equation}
Noting that $f(z)=f^s(z)=0$. So, by \eqref{eq-lambda1},
\begin{equation*}
\begin{split}
\lambda_1(G,B,\{z\})\leq \sigma_2(G\vee_zG).
\end{split}
\end{equation*}
This completes the proof of the theorem.
\end{proof}
We are now ready to prove Theorem \ref{thm-rigidity-2-sym}.
\begin{proof}[Proof of Theorem \ref{thm-rigidity-2-sym}]
The proof of the necessary part is the same as the proof of Theorem \ref{thm-main-rigidity-2}. We only need to prove the sufficient part.

Because $\wt G_x=\wt B_x=\{x\}$ for $x\in B$ and $\wt B_y=\emptyset $ for $y\in \Omega(G)\setminus \{z\}$,
\begin{equation}\label{eq-sigma2=sigma2}
\sigma_2(\wt G)=\sigma_2(G\vee_z\wt G_z).
\end{equation}
Then, by the monotonicity of Steklov eigenvalues and Theorem \ref{thm-sym},
\begin{equation}\label{eq-sigma2-min}
\begin{split}
\sigma_2(G\vee_z\wt G_z)\geq &\sigma_2((G\vee_z\wt G_z)\vee_z(G\vee_z\wt G_z))\\
=&\lambda_1(G\vee_z\wt G_z,B\cup \wt B_z,\{z\})\\
=&\min\{\lambda_1(G,B,\{z\}),\lambda_1(\wt G_z,\wt B_z,\{z\})\}.
\end{split}
\end{equation}
The last equality comes from the following easy fact
\begin{equation*}
\spec(G\vee_z\wt G_z,B\cup \wt B_z,\{z\})=\spec(G,B,\{z\})\sqcup\spec(\wt G_z,\wt B_z,\{z\}).
\end{equation*}
Moreover, by using the same fact and Theorem \ref{thm-sym} again,
\begin{equation}\label{eq-lambda-1}
\begin{split}
\lambda_1(G,B,\{z\})=&\lambda_1(\vee^r_z\Gamma, \sqcup^r\Sigma,\{z\})\\
=&\lambda_1(\Gamma,\Sigma,\{z\})\\
=&\sigma_2(\vee_z^2\Gamma)\\
\geq &\sigma_2(\vee_z^r\Gamma)\\
=&\sigma_2(G)
\end{split}
\end{equation}
where the last inequality comes from the monotonicity of Steklov eigenvalues and that $r\geq 2$. Now, combining \eqref{eq-lambda1-sigma2}, \eqref{eq-sigma2=sigma2}, \eqref{eq-sigma2-min} and \eqref{eq-lambda-1}, we know that
\begin{equation}
\sigma_2(\wt G)\geq \sigma_2(G).
\end{equation}
By monotonicity of Steklov eigenvalues again, $\sigma_2(\wt G)=\sigma_2(G)$. This completes the proof of the theorem.
\end{proof}
\section{Some applications on trees}\label{sec-app}
In this section, we give some applications of the monotonicity of Steklov eigenvalues on trees. All trees are considered as combinatorial trees with boundary without further indications.  First recall the definition of a star (See also \cite{FR}).
\begin{defn}
The tree obtained by identifying $r$ paths on one of their end vertices is called a star of degree $r$. The identified vertex is called the center of the star and the paths are called the arms of the star. The star of degree $r$ and with the lengths of the $r$ arms: $l_1,l_2,\cdots, l_r$ is denoted as $\St(l_1,l_2,\cdots,l_r)$. If $l_1=l_2=\cdots=l_r=l$, the star $\St(l_1,l_2,\cdots,l_r)$ is call a regular star and simply denoted as $\St(r;l)$.
\end{defn}
We next compute the eigenvalues and eigenspaces of a regular star.
\begin{lem}\label{lem-regular-star}
For $r\geq 2$ and $l\geq 1$,
\begin{equation*}
\sigma_2(\St(r;l))=\sigma_3(\St(r;l))=\cdots=\sigma_{r}(\St(r;l))=\frac1{l}.
\end{equation*}
Moreover, suppose the center of $\St(r;l)$ is $o$ and the $j^{\rm th}$ arm is: $o\sim v_{j 1}\sim v_{j 2}\sim\cdots\sim v_{j l}$ for $j=1,2,\cdots, r$. Then, the eigenspace of the eigenvalue $\frac{1}{l}$ is generated by
\begin{equation}
f_j(x)=\left\{\begin{array}{rl}0&x=o\\
\frac{k}{l}&x=v_{1k},k=1,2,\cdots,l\\
-\frac{k}{l}&x=v_{jk},k=1,2,\cdots,l\\
0&{\rm otherwise}
\end{array}\right.
\end{equation}
for $j=2,3,\cdots, r$.
\end{lem}
\begin{proof}
It is not hard to get the conclusion by direct verification.
\end{proof}
By using Lemma \ref{lem-regular-star} and the monotonicity of Steklov eigenvalues, we have the following estimate for finite trees containing a regular star as a subtree.
\begin{thm}\label{thm-reg-star}
For $r\geq 2$ and $l\geq 1$, let $\wt G$ be a finite tree containing $\St(r;l)$ as a subtree. Then,
\begin{equation}\label{eq-sigma-i-regular-star}
\sigma_i(\wt G)\leq\frac1l
\end{equation}
for $i=2,3,\cdots, r$. Moreover, the equality holds for $i=2$ if and only if $\wt G=\St(r;l)\vee_o\wt G_o$ with $\lambda_1\left(\wt G_o, L(\wt G_o)\setminus \{o\}, o\right)\geq \frac{1}{l}$. Here $o$ is the center of the star $\St(r;l)$.
\end{thm}
\begin{proof}
The inequality \eqref{eq-sigma-i-regular-star} comes from the monotonicity of Steklov eigenvalues and Lemma \ref{lem-regular-star}. Note that $\St(r;l)=\vee_o^rP$ with $r\geq 2$ where $P$ is a path of length $l$ with one of the end vertex $o$. So, the rigidity part comes from (3) and (4) of Corollary \ref{cor-rigidity-tree} by noting that $Z_1=\{o\}$ in this case because of Lemma \ref{lem-regular-star}.
\end{proof}
Next, we consider Steklov eigenvalues of a general star.
\begin{lem}\label{lem-star}
Let $r\geq 2$ and $1\leq l_1\leq l_2\leq\cdots\leq l_r$, and let $\sigma_2\leq \sigma_3\leq\cdots\leq\sigma_r$ be the positive Steklov eigenvalues for $\St(l_1,l_2,\cdots,l_r)$. Then,

${\rm(1)}$ $\frac1{\sigma_2},\frac1{\sigma_3},\cdots,\frac1{\sigma_r}$ are the $r-1$ roots of the polynomial
\begin{equation*}
\begin{split}
P(t)=&\sum_{i=1}^r(t-l_1)(t-l_2)\cdots\wh{(t-l_i)}\cdots (t-l_r)\\
=&\sum_{i=0}^{r-1}(-1)^{r-i-1}(i+1)p_{r-i-1}(l_1,l_2,\cdots,l_r)t^i.
\end{split}
\end{equation*}
Here $p_k(t_1,t_2,\cdots,t_r)$ is the elementary symmetric polynomial of degree $k$ on $t_1,t_2,\cdots,t_r$, for $k=0,1,2,\cdots,r$. More precisely, $p_0=1$ and
\begin{equation*}
p_k(t_1,t_2,\cdots, t_r)=\sum_{1\leq i_1<i_2<\cdots<i_k\leq r}t_{i_1}t_{i_2}\cdots t_{i_k}.
\end{equation*}

${\rm (2)}$ $Z\neq \emptyset $ if and only if
\begin{equation}\label{eq-Z-nonempty}
l_1=l_2=\cdots=l_{r-1}\ \mbox{and}\ l_r=rd+l_1
\end{equation}
for some nonnegative integer $d$. Moreover, in this case, $Z$ only contains the vertex on the $r^{\rm th}$ arm with distance $d$ to the center of the star, and
$$\sigma_2=\frac{1}{(r-1)d+l_1},\ \sigma_3=\sigma_4=\cdots=\sigma_{r}=\frac{1}{l_1}.$$ Here $Z$ is the same as in  Theorem \ref{thm-main-rigidity} for $G=\St(l_1,l_2,\cdots,l_r)$.
\end{lem}
\begin{proof}
(1) Let $v_i$ be the  end vertex of the $i^{\rm th}$ arm of the star for $i=1,2,\cdots,r$. Let $f$ be an eigenfunction for the eigenvalue $\sigma>0$ and $f(v_i)=x_i$ for $i=1,2,\cdots, r$. Note that $\Delta f|_\Omega=0$ where $\Omega$ means the interior of the star. So, the values of $f$ restricting on each arm must be an arithmetic progress. Thus, $$f(o)=(1-l_i\sigma)x_i.$$  Here $o$ is the center of the star. So, $\sigma$ is a positive Steklov eigenvalue of the star if and only if the linear system
\begin{equation}
\left\{\begin{array}{l}(1-l_1\sigma)x_1=(1-l_2\sigma)x_2=\cdots=(1-l_r\sigma)x_r\\
x_1+x_2+\cdots+x_r=0
\end{array}\right.
\end{equation}
has nonzero solutions. From this we get the conclusion.

(2) When $Z$ is nonempty, suppose that $Z$ contains a vertex $z$ on the $j^{\rm th}$ arm with distance $d$ to the center of the star. Let $f\in \R^{V}$ be  harmonic in the interior and $f(v_i)=x_i$ for $i=1,2,\cdots,r$ with $x_1+x_2+\cdots+x_r=0$. Then $f(z)=0$ since $z\in Z$. Let $y$ be the value of $f$ at the center of the star. Because the values of $f$ on each arm form an arithmetic progress, the value of $f$ at the vertex adjacent to the center on the $i^{\rm th}$ arm must be $y-\frac{y-x_i}{l_i}$. Moreover, because $f$ is harmonic at the center, we have
\begin{equation}\label{eq-har-st}
\sum_{i=1}^r\frac{y-x_i}{l_i}=0.
\end{equation}
Furthermore, by that $f(z)=0$, we have
\begin{equation}\label{eq-y}
y=-\frac{d}{l_j-d}x_j.
\end{equation}
Substituting \eqref{eq-y} into \eqref{eq-har-st}, we know that
\begin{equation}
\sum_{i\neq j}\frac{x_i}{l_i}+\left[\left(\sum_{i=1}^r\frac{d}{(l_j-d)l_i}\right)+\frac1{l_j}\right]x_j=0
\end{equation}
whenever $x_1+x_2+\cdots+x_r=0$. Thus
\begin{equation}
\frac1{l_1}=\frac1{l_2}=\cdots=\frac1{l_{j-1}}=\left(\sum_{i=1}^r\frac{d}{(l_j-d)l_i}\right)+\frac1{l_j}=\frac{1}{l_{j+1}}=\cdots=\frac{1}{l_{r}}.
\end{equation}
So, the $j^{\rm th}$ arm must be the longest arm and without loss of generality, we can assume $j=r$. Furthermore, $l_1=l_2=\cdots=l_{r-1}$ and $l_r=dr+l_1$.  The eigenvalues can be easily obtained by using (1).
\end{proof}
Combining Lemma \ref{lem-star} and the monotonicity of Steklov eigenvalues, we have the following estimates for trees containing a star as a subtree.
\begin{thm}
Let $\wt G$ be a finite tree  containing the star $\St(l_1,l_2,\cdots,l_r)$ with $r\geq 2$ and $1\leq l_1\leq l_2\leq\cdots\leq l_r$ as a subtree. Then,

${\rm (1)}$ for $k=1,2,\cdots,r-1$
\begin{equation*}
p_k\left(\frac{1}{\sigma_2(\wt G)},\frac{1}{\sigma_3(\wt G)},\cdots,\frac{1}{\sigma_r(\wt G)}\right)\geq\frac{r-k}{r}p_{k}(l_1,l_2,\cdots,l_r)
\end{equation*}
and
\begin{equation*}
p_k\left(\sigma_2(\wt G),\sigma_3(\wt G),\cdots,\sigma_r(\wt G)\right)\leq \frac{(k+1)p_{r-k-1}(l_1,l_2,\cdots,l_r)}{p_{r-1}(l_1,l_2,\cdots,l_r)};
\end{equation*}

${\rm (2)}$ when $l_1,l_2,\cdots,l_r$ do not satisfy \eqref{eq-Z-nonempty}, the equality holds for some inequalities in (1) if and only if $\wt G=\St(l_1,l_2,\cdots,l_r)$;

${\rm (3)}$ when $l_1,l_2,\cdots, l_r$ satisfy \eqref{eq-Z-nonempty} and the equality holds for some inequalities in (1), $\wt G=\St(l_1,l_2,\cdots,l_r)\vee_z\wt G_z$ with $\lambda_1(\wt G_z,L(\wt G_z)\setminus\{z\},\{z\})\geq \frac{1}{l_1}$. Here $z$ is vertex in $Z$ as described in (2) of Lemma \ref{lem-star};

${\rm (4)}$ $\sigma_2(\wt G)\leq \frac{r}{l_1+l_2+\cdots l_r}$. Moreover, when $r\geq 3$, the equality holds if and only $l_1=l_2=\cdots=l_r$ and $\wt G=\St(r;l_1)\vee_o\wt G_o$ with $\lambda_1(\wt G_o,L(\wt G_o)\setminus\{o\},\{o\})\geq \frac{1}{l_1}$. Here $o$ is the center of the star.
\end{thm}
\begin{proof}
Let $\sigma_2\leq \sigma_3\leq \cdots\leq \sigma_r$ be the positive Steklov eigenvalues of $\St(l_1,l_2,\cdots,l_r)$.

 (1) By the monotonicity of Steklov eigenvalues, $\sigma_i(\wt G)\leq \sigma_i$ for $i=2,3,\cdots, r$. Then, by Lemma \ref{lem-star},
\begin{equation*}
p_k\left(\frac{1}{\sigma_2(\wt G)},\frac{1}{\sigma_3(\wt G)},\cdots,\frac{1}{\sigma_r(\wt G)}\right)\geq p_k\left(\frac{1}{\sigma_2},\frac{1}{\sigma_3},\cdots,\frac{1}{\sigma_r}\right)=\frac{r-k}{r}p_{k}(l_1,l_2,\cdots,l_r)
\end{equation*}
for $k=1,2,\cdots, r-1$. The other inequalities can be obtained similarly.

(2) If the equality holds for some inequalities in (1), then $\sigma_i(\wt G)=\sigma_i$ for $i=1,2,\cdots, r$. Moreover, in the case that $l_1,l_2,\cdots, l_r$ do not satisfy \eqref{eq-Z-nonempty}, $Z=\emptyset$ by Lemma \ref{lem-star}. So, by (1) in Corollary \ref{cor-rigidity-tree}, $\wt G=\St(l_1,l_2,\cdots,l_r)$.

(3) It follows from (2) of Corollary \ref{cor-rigidity-tree} and Lemma \ref{lem-star}.

(4) By (1), we have
\begin{equation}
\frac{r-1}{\sigma_2(\wt G)}\geq \sum_{i=2}^r\frac{1}{\sigma_i(\wt G)}\geq \frac{r-1}{r}(l_1+l_2+\cdots+l_r).
\end{equation}
Thus $$\sigma_2(\wt G)\leq\frac{r}{l_1+l_2+\cdots+l_r}.$$

Moreover, when $r\geq 3$, if the equality of the inequality holds, then
\begin{equation}
\sigma_2(\wt G)=\sigma_{3}(\wt G)=\cdots=\sigma_{r}(\wt G)=\sigma_2=\sigma_3=\cdots=\sigma_r=\frac{r}{\sum_{i=1}^rl_i}.
\end{equation}
Then, by Lemma \ref{lem-star},
\begin{equation}
\begin{split}
\frac{(r-1)(r-2)}{2}\left(\frac{\sum_{i=1}^rl_i}{r}\right)^2=\sum_{2\leq i<j\leq r}\frac{1}{\sigma_i\sigma_j}=\frac{r-2}{r}\sum_{1\leq i<j\leq r}l_il_j
\end{split}
\end{equation}
Thus
\begin{equation}
0=(r-1)\left(\sum_{i=1}^rl_i\right)^2-2r\sum_{1\leq i<j\leq r}l_il_j=\sum_{1\leq i<j\leq r}(l_i-l_j)^2.
\end{equation}
Hence $l_1=l_2=\cdots=l_r$. The other conclusion then follow from Theorem \ref{thm-reg-star}.
\end{proof}
We next recall the definition of a regular comb (See also \cite{FR}).
\begin{defn}
A comb over a path of length $r$ such that each tooth is a path of length $l$ with one of the end vertices of the tooth on the base is called a regular comb with $r+1$ teeth of length $l$ and is denoted as $\Comb(r;l)$.
\end{defn}
\begin{lem}\label{lem-comb}
Let $0=\mu_1<\mu_2\leq\cdots\leq\mu_{r+1}$ be the Laplacian eigenvalues of the path of length $r$. Then, for any $l\geq 1$, $$\sigma_i(\Comb(r;l))=\frac{\mu_i}{1+\mu_il}$$
for $i=1,2,\cdots, r+1$. Moreover, suppose the base of $\Comb(r;l)$ is the path $P$: $v_{0,0}\sim v_{0,1}\sim\cdots\sim v_{0,r}$, and the tooth on $v_{0,i}$ is the path $T_i$: $v_{0,i}\sim v_{1,i}\sim \cdots \sim v_{l,i}$ for $i=0,1,2,\cdots, r$. Then, the eigenfunction for $\sigma_i(\Comb(r;l))$ is
\begin{equation}
f_i(v_{j,k})=\varphi_i(v_{0,k})(1+j\mu_i)
\end{equation}
for $j=0,1,\cdots, l$ and $k=0,1,\cdots,r$, where $\varphi_i$ is the Laplacian eigenfunction on $P$ for $\mu_i$. In fact, $\mu_i=2-2\cos\frac{(i-1)\pi}{r+1}$, and $\varphi_i(v_{0,j})=\cos\frac{(i-1)(2j+1)\pi}{2(r+1)}$ for $j=0,1,\cdots,r$ and $i=1,2,\cdots,r+1$.
\end{lem}
\begin{proof}
Let $\Delta$ be Laplacian operator on $\Comb(r;l)$. Then,
\begin{equation}
\begin{split}
\Delta f_i(v_{0,0})=&(f_i(v_{0,1})-f_i(v_{0,0}))+(f_i(v_{1,0})-f_i(v_{0,0}))\\
=&(\varphi_i(v_{0,1})-\varphi_i(v_{0,0}))+\mu_i\varphi_i(v_{0,0})\\
=&\Delta_P\varphi_i(v_{0,0})+\mu_i\varphi_i(v_{0,0})\\
=&0,
\end{split}
\end{equation}

\begin{equation}
\begin{split}
\Delta f_i(v_{0,j})=&(f_i(v_{0,j+1})-f_i(v_{0,j}))+(f_i(v_{1,j})-f_i(v_{0,j}))+(f_i(v_{0,j-1})-f_i(v_{0,j}))\\
=&\Delta_P\varphi_i(v_{0,j})+\mu_i\varphi_i(v_{0,j})\\
=&0
\end{split}
\end{equation}
for $j=2,3,\cdots, r-1$, and similarly,
\begin{equation}
\Delta f_i(v_{0,r})=0.
\end{equation}
It is also clear that $\Delta f_i(v_{j,k})=0$ for $j=1,2,\cdots,l-1$ and $k=0,1,\cdots,r$. It is then straightforward to verify that
\begin{equation}
\frac{\p f_i}{\p n}(v_{l,k})=\frac{\mu_i}{1+l\mu_i}f_i(v_{l,k})
\end{equation}
for $k=0,1,\cdots,r$. This completes the proof of the lemma.  For the eigenvalues and eigenfunctions for a path, see \cite[P. 9]{BR}.
\end{proof}
By combining Lemma \ref{lem-comb} and the monotonicity of Steklov eigenvalues, we have the following estimates for trees containing a regular comb as a subtree.
\begin{thm}
Let $\wt G$ be a tree containing the regular comb $\Comb(r;l)$ with $r,l\geq 1$ as a subtree. Then,
\begin{equation}
\sigma_i(\wt G)\leq \frac{4\sin^2\frac{(i-1)\pi}{2(r+1)}}{1+4l\sin^2\frac{(i-1)\pi}{2(r+1)}}
\end{equation}
for $i=1,2\cdots, r+1$. Moreover, the equality holds for all $i=1,2,\cdots, r+1$ if and only if $\wt G=\Comb(r;l)$.
\end{thm}
\begin{proof} The inequality comes from monotonicity of Steklov eigenvalues and Lemma \ref{lem-comb}. The rigidity comes from the fact that $Z=\emptyset$ in this case and (1) of Corollary \ref{cor-rigidity-tree}.
\end{proof}
\begin{rem}
Because a star of degree two and a regular comb with two teeth are both paths, the estimates above for trees containing a star or a regular comb as a subtree can be viewed as generalizations of the isodiametric estimate in He-Hua \cite{He-Hua1}.
\end{rem}
Finally, we consider trees contained in or containing a ball of a homogeneous tree. We use $T(r,d)$ to denote the induced subgraph of the homogeneous tree of degree $d$ on a ball of radius $r$. Moreover, for a graph $G$ and $x\in V(G)$, we denote the set of vertices in $V(G)$ with distance $r$ to $x$ as $S_x(r)$. We first compute the Steklov spectrum of $T(r,d)$.
\begin{lem}\label{lem-tree}
Let $r\geq 1$ and $d\geq 3$. Then, $T(r,d)$ has $r+1$ distinct Steklov eigenvalues:
$$0=\ol\sigma_1<\ol\sigma_2<\cdots<\ol\sigma_{r+1}$$
where
\begin{equation}
\ol\sigma_k=\frac{1}{\sum_{i=0}^{r+1-k}(d-1)^{i}}=\frac{d-2}{(d-1)^{r+2-k}-1}
\end{equation}
for $k=2,3,\cdots, r+1$. Moreover, when $k\geq 3$, the eigenspace of $\ol\sigma_k$ is of dimension $(d-2)d(d-1)^{k-3}$ and generated by the following eigenfunctions: Let $x\in S_o(k-2)$, $y_1,y_2,\cdots,y_{d-1}\in S_o(k-1)\cap S_x(1)$, for each $\alpha=2,\cdots,d-1$, define
\begin{equation*}
f_{xy_1y_\alpha}(v)=\left\{\begin{array}{ll}\frac{\sum_{i=r+1-k-j}^{r+1-k}(d-1)^i}{\sum_{i=0}^{r+1-k}(d-1)^i}& v\in S_{y_1}(j)\cap S_o(j+k-1),j=0,1,\cdots,r-k+1\\
0&v=x\\
-\frac{\sum_{i=r+1-k-j}^{r+1-k}(d-1)^i}{\sum_{i=0}^{r+1-k}(d-1)^i}& v\in S_{y_\alpha}(j)\cap S_o(j+k-1), j=0,1,\cdots,r-k+1\\
0&{\mbox{otherwise}}
\end{array}\right.
\end{equation*}
The eigenspace of $\ol \sigma_2$ is of dimension $d-1$ and generated by the following eigenfunctions: Let $x_1,x_2,\cdots x_d\in S_o(1)$, for each $\alpha=2,3,\cdots, d$, define
\begin{equation*}
f_{x_1x_\alpha}(v)=\left\{\begin{array}{ll}\frac{\sum_{i=r-1-j}^{r-1}(d-1)^i}{\sum_{i=0}^{r-1}(d-1)^i}& v\in S_{x_1}(j)\cap S_o(j+1),j=0,1,\cdots,r-1\\
0&v=o\\
-\frac{\sum_{i=r-1-j}^{r-1}(d-1)^i}{\sum_{i=0}^{r-1}(d-1)^i}& v\in S_{x_\alpha}(j)\cap S_o(j+1),j=0,1,\cdots,r-1\\
0&{\mbox{otherwise.}}
\end{array}\right.
\end{equation*}
Here $o$ is the root of $T(r,d)$.
\end{lem}
\begin{proof}
The conclusion of the lemma follows by direct verification.
\end{proof}
By applying Lemma \ref{lem-tree} and the monotonicity of Steklov eigenvalues, we have the following results.
\begin{thm}Let $G$ be a nontrivial finite tree.

${\rm (1)}$ If $G$ is a subtree of $T(r,d)$ with $r\geq 1$ and $d\geq 3$, then
\begin{equation*}
\sigma_j(G)\geq \ol\sigma_k=\frac{d-2}{(d-1)^{r+2-k}-1}
\end{equation*}
for $j=2,3,\cdots,|L(G)|$, when $d(d-1)^{k-3}<j\leq d(d-1)^{k-2}$ for $k=2,3,\cdots,r+1$.

${\rm(2)}$ If $\wt G$ contains $T(r,d)$ with $r\geq 1$ and $d\geq 3$ as a subtree,  then
\begin{equation*}
\sigma_j(\wt G)\leq\ol\sigma_k=\frac{d-2}{(d-1)^{r+2-k}-1}
\end{equation*}
for $j=2,3,\cdots, d(d-1)^{r-1}$, when $d(d-1)^{k-3}<j\leq d(d-1)^{k-2}$ for $k=2,3,\cdots,r+1$. Moreover, the equality holds for $j=2$ if and only if $\wt G=T(r,d)\vee_o\wt G_o$ with $$\lambda_1(\wt G_o, L(\wt G_o)\setminus\{o\},\{o\})\geq \ol\sigma_2=\frac{d-2}{(d-1)^{r}-1}.$$
Here $o$ is the root of the tree $T(r,d)$.
\end{thm}
\begin{proof}
(1) The conclusion follows directly from Lemma \ref{lem-tree} and monotonicity of Steklov eigenvalues.

(2) The inequality comes from Lemma \ref{lem-tree} and monotonicity of Steklov eigenvalues directly. The rigidity follows directly from (3) and (4) of Corollary \ref{cor-rigidity-tree} and that $Z_1=\{o\}$ by Lemma \ref{lem-tree}. Here $Z_1$ is the same as in Theorem \ref{thm-main-rigidity-2}.
\end{proof}
\begin{rem}
The inequality in (1) of the last theorem for $j=2$ is also mentioned in He-Hua \cite{He-Hua2} where they also obtained the rigidity of the inequality for $j=2$.
\end{rem}

\end{document}